\documentclass[reqno,10pt]{amsart}
\usepackage{amsmath}
\usepackage{amsfonts}
\usepackage{amssymb}
\usepackage{amsthm}
\usepackage{newlfont}
\usepackage{graphicx}
\newtheorem{thm}{Theorem}[section]
\newtheorem{cor}[thm]{Corollary}
\newtheorem{lem}[thm]{Lemma}

\theoremstyle{definition}

\theoremstyle{remark}
\newtheorem{rem}[thm]{Remark}
\numberwithin{equation}{section}

\newcommand{\abs}[1]{\left\vert#1\right\vert}

\newcommand{\A}{\mathcal{A}}
\newcommand{\sgn}{\text{sgn}}

\newcommand{\ed}{\end {document}}

\newcommand{\nb}{|\nabla|}

\title[maximum principle for a transport-diffusion model]
{On a generalized maximum principle for a transport-diffusion model with
$\log$-modulated fractional dissipation}
\author[H. Dong]{Hongjie Dong}
\address[H. Dong]{Division of Applied Mathematics, Brown University,
182 George Street, Providence, RI 02912, USA}
\email{Hongjie\_Dong@brown.edu}
\author[D. Li]{Dong Li}
\address[D. Li]{Department of Mathematics, University of British Columbia, Vancouver BC Canada V6T 1Z2}%
\email{mpdongli@gmail.com}

\begin{document}
\begin{abstract}
We consider a transport-diffusion equation of the form $\partial_t
\theta +v \cdot \nabla \theta + \nu \A \theta =0$, where $v$ is a
given time-dependent vector field on $\mathbb R^d$. The operator
$\A$ represents log-modulated fractional dissipation: $\A=\frac
{|\nabla|^{\gamma}}{\log^{\beta}(\lambda+|\nabla|)}$ and the
parameters $\nu\ge 0$, $\beta\ge 0$, $0\le \gamma \le 2$,
$\lambda>1$. We introduce a novel nonlocal decomposition of the
operator $\A$ in terms of a weighted integral of the usual
fractional operators $|\nabla|^{s}$, $0\le s \le \gamma$ plus a
smooth remainder term which corresponds to an $L^1$ kernel. For a
general vector field $v$ (possibly non-divergence-free) we prove a
generalized $L^\infty$ maximum principle of the form $ \|
\theta(t)\|_\infty \le e^{Ct} \| \theta_0 \|_{\infty}$ where the
constant $C=C(\nu,\beta,\gamma)>0$. In the case $\text{div}(v)=0$ the same
inequality holds for $\|\theta(t)\|_p$ with $1\le p \le \infty$. At
the cost of an exponential factor, this extends a recent result of
Hmidi \cite{Hmidi11} to the full regime $d\ge 1$, $0\le \gamma \le
2$ and removes the incompressibility assumption in the $L^\infty$
case.
\end{abstract}

\maketitle

\section{Introduction}

We consider the transport equation with $\log$-modulated fractional
dissipation of the form
\begin{align} \label{e1}
\begin{cases}
\partial_t \theta + v \cdot \nabla \theta + \nu
\A\theta =0,\quad (t,x) \in (0,\infty) \times \mathbb R^d, \\
\A\theta:=\frac { \nb^\gamma} { \log^\beta (\lambda +|\nabla|) } \theta,\\
\theta(0,x)=\theta_0,
\end{cases}
\end{align}
where $\nu\ge 0$, and $v=v(t,x): \; [0,\infty)\times \mathbb R^d \to \mathbb R^d$ is
a given vector field, possibly non-divergence-free. The basic unknown is the scalar function $\theta=\theta(t,x)$ which
is usually termed ``active scalar''.
The operator $\A$ is defined via the Fourier transform
\begin{align*} 
\widehat{\A f}(\xi) = \frac { |\xi|^{\gamma} } { \log^{\beta} (\lambda +|\xi|)} \hat f(\xi),
\qquad \xi \in \mathbb R^d,
\end{align*}
where the parameters $0\le \gamma \le 2$, $\beta\ge 0$, $\lambda>1$. It is termed $\log$-modulated fractional
dissipation since it is the usual fractional Laplacian operator $|\nabla|^{\gamma}$ divided by a
logarithm symbol. Along the way we will also consider
 a variant of the operator $\A$ which is denoted as $\A_1 = |\nabla|^{\gamma}
\log^{-\beta}(\lambda-\Delta)$, i.e.
\begin{align*} 
\widehat{\A_1 f}(\xi) = \frac { |\xi|^{\gamma} } { \log^{\beta}
(\lambda +|\xi|^2)} \hat f(\xi), \qquad \xi \in \mathbb R^d.
\end{align*}
The main objective of this paper is to prove some maximum principles for the operator $\A$ and $\A_1$ in
Lebesgue 
 spaces.

The transport-diffusion model \eqref{e1}  is a natural generalization of several linear and nonlinear fluid equations
such as the two-dimensional surface quasi-geostrophic equations, fractional Burgers equations,
vortex patch models, and Boussinesq systems. See, for instance,
the recent work \cite{Ta09, DDL09, CCS10, Wu11, DKV, DKSV} and references therein.
In these problems, the velocity $v$ is typically related to the active scalar $\theta$ by a constitutive relation
$v= \mathcal T(\theta)$ where $\mathcal T$ could be some singular integral operator or more generally a nonlocal operator.
To obtain local and global wellposedness results for the nonlinear problems, an important first step is to get a priori
$L^p$, $1\le p \le \infty$ estimates of solutions.  Specific to the linear problem \eqref{e1}, one needs to prove
  $L^p$ bounds on the active scalar $\theta$ \textit{independent
of the size of $v$}.  We refer to these types of results as $L^p$
maximum principle estimates. In this respect, the two-dimensional
dissipative surface quasi-geostrophic equations can be regarded as a
(nonlinear) version of \eqref{e1} and they correspond to the case
$\beta=0$ in the operator $\A$. A classical result is due to A.
C\'ordoba and D. C\'ordoba \cite{CC04}  who proved the following
\begin{align*}
\| \theta(t) \|_p \le \| \theta(0) \|_p, \quad \forall\, t\ge0, \,
1\le p\le \infty.
\end{align*}
In a recent article, Hmidi \cite{Hmidi11} initiated the study of
\eqref{e1} and obtained the following important maximum principle:

\begin{thm}[Hmidi \cite{Hmidi11}] \label{thm0}
Let the dimension $d=2,3$ and let $\nu\ge 0$, $0\le \gamma \le 1$,
$\beta \ge 0$, $\lambda \ge e^{\frac{3+2\alpha}{\beta}}$. Assume the
velocity $v$ is divergence-free, i.e. $\nabla \cdot v =0$. Then any
smooth solution of \eqref{e1} satisfies
\begin{align*}
\| \theta (t) \|_p \le \| \theta (0) \|_p, \quad \forall\, t\ge 0,\,
1\le p \le \infty.
\end{align*}

\end{thm}

To prove Theorem \ref{thm0}, Hmidi used the theory of
$C_0$-semigroup of contractions on $L^p$ ($1<p<\infty$) for the
family of convolution kernels $(K_t)_{t\ge 0}$ defined by
\begin{align*}
\widehat{K_t}(\xi) =e ^{-t \frac
{|\xi|^{\gamma}}{\log^{\beta}(\lambda+|\xi|)}}.
\end{align*}
The key step is to get the positivity of the kernel $K_t$. For this
purpose, Hmidi used the Askey's criterion for characteristic
functions \cite{Askey}.  The restrictions on the dimension $d$ and
the parameters $(\gamma,\beta, \lambda)$ are mainly due to the use
of this criterion. Hmidi conjectured that the maximum principle
should hold for all dimensions $d\ge 1$ and the full range $0\le
\gamma \le 2$ and $\beta \ge 0$. The purpose of this paper is to
give an affirmative answer to this question at the cost of a
harmless exponential factor.

\begin{thm}[Generalized maximum principle, $L^\infty$ case] \label{thm1}
Let $\nu\ge 0$, $d\ge 1$, $0\le \gamma \le 2$ and $\beta\ge 0$,
$\lambda > 1$.
Assume $\theta=\theta(t,x)$ is a smooth solution of
\eqref{e1}  which decays at spatial infinity, i.e., for any fixed $t\ge 0$,
\begin{align}
\lim_{|x|\to \infty} \theta(t,x) = 0, \label{decay_cond}
\end{align}
Then we have
\begin{align}
\| \theta(t) \|_{\infty} \le e^{Ct} \| \theta_0 \|_{\infty},\quad \forall\, t \ge 0, \label{emax_1}
\end{align}
where $C>0$ is a constant depending only on $(\nu,d,\gamma,\beta, \lambda)$.

\end{thm}

\begin{rem}
The same result holds if we replace the dissipation operator $\A$ by $\A_1$ in \eqref{e1}.
As we shall see in Section 2, the proof for $\A_1$ case is actually simpler.
The decay condition \eqref{decay_cond} is fairly weak as most smooth solutions to these type
of fluid equations typically belong to the Sobolev space $C_t^0 H_x^s$ which can easily imply
\eqref{decay_cond}. We should also stress that we do not assume any divergence-free
condition on $v$ in Theorem \ref{thm1}. This can have applications for compressible fluid
equations.
\end{rem}

To prove Theorem \ref{thm1}, we shall use a completely new idea
which avoids the use of Askey's criterion. Namely we introduce a
novel nonlocal decomposition of the operator $\A$ (see Section
\ref{sec2} for more details) in terms of a weighted integral of the
usual fractional operators $|\nabla|^{s}$, $0\le s \le \gamma$ plus
a smooth remainder term which corresponds to an $L^1$ kernel. Thanks
to this new decomposition, we shall only need to appeal to the
classic maximum principle for the fractional Laplacian operators.
In a similar vein, one can even consider a weighted integral of a parameterized family of
nonlocal operators each of which obeys a maximum principle. However we shall not pursue
this generality here.

As was already mentioned, Theorem \ref{thm1} deals with the $L^\infty$
norm and no special assumption is needed on the velocity field $v$.
On the other hand for more general $L^p$-norms with $1\le p<\infty$,
the divergence-free condition on the vector field $v$ has to be
assumed, as one needs to calculate the time derivative of the $L^p$ norm and perform
integration by parts.

\begin{thm}[Generalized maximum principle, $L^p$ case] \label{thm2}
Let $\nu\ge 0$, $d\ge 1$, $0\le \gamma \le 2$ and $\beta\ge 0$,
$\lambda >1$. Assume the vector field $v=v(t,x)$ is divergence free, i.e.
$\nabla \cdot v =0$. If $\theta=\theta(t,x)$ is a smooth solution
of \eqref{e1}, then for any $1\le p \le \infty$, we have
\begin{align}
\| \theta(t) \|_{p} \le e^{Ct} \| \theta_0 \|_{p}, \quad \forall\, t>0, \label{emax_2}
\end{align}
where $C>0$ is a constant depending only on $(\nu,d,\gamma,\beta,
\lambda)$.
\end{thm}

\begin{rem}
Both Theorem \ref{thm1} and Theorem \ref{thm2} hold in the periodic boundary condition
case.\footnote{We thank Edriss Titi for this suggestion.} In the periodic setting, the
decay condition \eqref{emax_1} is no longer needed as the periodic domain is compact.
\end{rem}

It is an interesting question whether one can prove the sharp constant is $C=0$ in
both Theorem \ref{thm1} and Theorem \ref{thm2}. We conjecture this is indeed the
case at least for a generic set of parameters.

%
%
%

The rest of this article is organized as follows. In Section 2 we
introduce the nonlocal decomposition for both the operator $\mathcal
A$ and the operator $\mathcal A_1$. In Section 3 we give the proof
for Theorem \ref{thm1} and Theorem \ref{thm2} for the operator
$\mathcal A_1$. The case $0\le \gamma \le 1$ of $\mathcal A$ is also
covered there. In Section 4 we complete the proof of the main
theorems for the operator $\mathcal A$ in the regime $1<\gamma \le
2$.

We conclude the introduction by setting up some

\subsubsection*{Notations}
\begin{itemize}
\item For any two quantities $X$ and $Y$, we denote $X \lesssim Y$ if
$X \le C Y$ for some constant $C>0$. Similarly $X \gtrsim Y$ if $X
\ge CY$ for some $C>0$. We denote $X \sim Y$ if $X\lesssim Y$ and $Y
\lesssim X$. We shall write $X\lesssim_{Z_1,Z_2,\cdots, Z_k} Y$ if
$X \le CY$ and the constant $C$ depends on the quantities
$(Z_1,\cdots, Z_k)$. Similarly we define $\gtrsim_{Z_1,\cdots, Z_k}$
and $\sim_{Z_1,\cdots,Z_k}$.

\item For any $f$ on $\mathbb R^d$, we denote the Fourier transform of
$f$ has
\begin{align*}
(\mathcal F f)(\xi) = \hat f (\xi) = \int_{\mathbb R^d} f(x) e^{-i \xi \cdot x}\,dx.
\end{align*}
The inverse Fourier transform of any $g$  is given by
\begin{align*}
(\mathcal F^{-1} g)(x) = \frac 1 {(2\pi)^d} \int_{\mathbb R^d}  g(\xi) e^{i x \cdot \xi} \,d\xi.
\end{align*}

\item For any real number $x$, the sign function $\sgn (x)$ is defined as follows
\begin{align*}
\text{sgn}(x)=
\begin{cases}
-1, \quad \text{if $x<0$}, \\
0, \quad \text{if $x=0$}, \\
1, \quad \text{if $x>0$}.
\end{cases}
\end{align*}
For any complex $z$ with $\text{Re}(z)>0$, the Gamma function $\Gamma(z)$ is given by the expression
\begin{align*}
\Gamma(z) = \int_0^{\infty} t^{z-1} e^{-t}\, dt.
\end{align*}

\item
 We will also occasionally need to use the
Littlewood-Paley frequency projection operators. Let $\varphi(\xi)$
be a smooth bump function supported in the ball $|\xi| \leq 2$ and
equal to one on the ball $|\xi| \leq 1$. For each dyadic number $N
\in 2^{\mathbb Z}$ we define the Littlewood-Paley operators
\begin{align*}
\widehat{P_{\leq N}f}(\xi) &:=  \varphi(\xi/N)\hat f (\xi), \notag\\
\widehat{P_{> N}f}(\xi) &:=  [1-\varphi(\xi/N)]\hat f (\xi), \notag\\
\widehat{P_N f}(\xi) &:=  [\varphi(\xi/N) - \varphi (2 \xi /N)] \hat
f (\xi). 
\end{align*}
Similarly we can define $P_{<N}$, $P_{\geq N}$, and $P_{M < \cdot
\leq N} := P_{\leq N} - P_{\leq M}$, whenever $M$ and $N$ are dyadic
numbers.

\end{itemize}

\subsection*{Acknowledgements}
H. Dong was partially supported by the NSF under agreements DMS-0800129 and DMS-1056737.

\section{The nonlocal decomposition} \label{sec2}

We start with the following lemma which establishes the nonlocal decomposition
of the log-modulated fractional dissipation operator $\frac{ \nb^{\gamma} } { \log^{\beta} ( \lambda+\nb)}$
 in the regime $0\le \gamma \le 1$. One should notice the subtle difference between this operator and
 the operator
 $\A_1$ in the logarithmic term. By a simple change of variable $|\xi| \to |\xi|^2$, the decomposition
 of the operator $\A_1$ is addressed in the next corollary.
 After that we establish the decomposition for the operator $\A$ in the regime $1<\gamma\le 2$.
 The proof will be more involved due to certain first order negative corrections.

\begin{lem}[Nonlocal decomposition, case $0\le \gamma\le 1$]
\label{lem1}
Let  $d\ge 1$, $0\le \gamma \le 1$ and $\beta> 0$, $\lambda >1$.
 Then we have the
decomposition:
\begin{align} \label{e200}
\frac { \nb^{\gamma}} {\log^{\beta}(\lambda+|\nabla|)} = C_{\beta} \int_0^{\gamma}
\tau^{\beta-1} |\nabla|^{ {\gamma} -\tau} \,d\tau + P,
\end{align}
where $P$ is a smooth Fourier multiplier which maps $L^p$ to $L^p$ for all $1\le p \le +\infty$.
More precisely, for any function $f$
\begin{align*}
(Pf)(x) = (K*f)(x) = \int K(x-y) f(y) \,dy,
\end{align*}
and $\| K\|_{L_x^1} \le Const$.
\end{lem}

\begin{proof}[Proof of Lemma \ref{lem1}]
On the Fourier side the identity \eqref{e200} is equivalent to the following
\begin{align} \label{e201}
\frac { |\xi| ^\gamma} { \log^\beta (\lambda+|\xi|) }  = C_{\beta}
\int_0^{\gamma} \tau^{\beta-1} |\xi|^{ {\gamma}-\tau} \,d\tau +P(\xi).
\end{align}

To show \eqref{e201} we start with the simple identity
\begin{align} \label{log_simple}
\frac 1 { \log^\beta (\lambda+|\xi| )} = \frac 1 {\Gamma (\beta)}
\int_0^{\infty} \tau^{\beta-1} (\lambda+ |\xi|)^{-\tau} \,d\tau.
\end{align}

Hence
\begin{align*}
 \frac {(\lambda+ |\xi|)^{ {\gamma}} } { \log^{\beta} (\lambda+ |\xi|) }
= & {\frac 1 {\Gamma (\beta)}}\int_0^{{\gamma}} \tau^{\beta-1} (\lambda + |\xi|)^{\gamma - \tau} \,d\tau \notag \\
& \qquad + {\frac 1 {\Gamma (\beta)}}\int_{{\gamma} }^\infty
\tau^{\beta-1} (\lambda + |\xi|)^{\gamma - \tau} \,d\tau.
\end{align*}

We then set $C_{\beta} = \frac 1 {\Gamma(\beta)}$ and obtain \eqref{e201} with
\begin{align}
P(\xi) & =  C_{\beta} \int_0^{\gamma} \tau^{\beta-1}
\Bigl( (\lambda+|\xi|)^{{\gamma} -\tau} - |\xi|^{{\gamma}-\tau} \Bigr) \,d\tau  \label{330a} \\
& \qquad +C_{\beta} \int_{\gamma}^\infty \tau^{\beta-1} (\lambda+ |\xi|)^{\gamma -\tau} \,d\tau \label{330b} \\
& \qquad + \frac { |\xi|^{\gamma} - (\lambda+ |\xi|)^{ {\gamma}} } {\log^{\beta} ( \lambda +|\xi|) }.
\label{330c}
\end{align}

It remains for us to show the $L^1$ boundedness of $\mathcal F^{-1} (P)$. We first deal with the
piece \eqref{330a}. By the Fundamental Theorem of Calculus, we have for any $0\le t \le 1$,
\begin{align}
(\lambda+\abs{\xi})^{t} - |\xi|^t = \int_{0}^{\lambda}  t (s+\abs{\xi})^{t-1} \,ds. \label{402a}
\end{align}
If $0<t<1$, then
\begin{align*}
(s+\abs{\xi})^{t-1} = \frac 1 {\Gamma(1-t)} \int_{0}^{\infty} y^{-t} e^{-y (s+|\xi|)} \,dy.
\end{align*}
Since for $y>0$ the Poisson kernel $\mathcal F^{-1} (e^{-y|\xi|} )$ is positive,  it follows easily
 that $\mathcal F^{-1} ( (s+\abs{\xi})^{t-1} )$ is a non-negative function and furthermore,
\begin{align*}
\| \mathcal F^{-1} ( (s+\abs{\xi})^{t-1} ) \|_{L_x^1} = s^{t-1}.
\end{align*}
By \eqref{402a}, we get
\begin{align}
\| \mathcal F^{-1} ( (\lambda+\abs{\xi})^{t} - |\xi|^t )  \|_{L_x^1} =\lambda^t. \label{423a}
\end{align}
Plugging the above estimate into \eqref{330a}, we get
\begin{align*}
\| \mathcal F^{-1} \eqref{330a}  \|_{L_x^1} \le C_{\beta}
\int_0^{\gamma} \tau^{\beta-1} \lambda^{\gamma-\tau} \,d\tau
 < \infty,
\end{align*}
which is clearly good for us.

For \eqref{330b}, we just note that for $\tau>\gamma$,
\begin{align*}
\| \mathcal F^{-1} ( (\lambda+|\xi|)^{\gamma-\tau} ) \|_{L_x^1} = \lambda^{\gamma-\tau},
\end{align*}
and hence
\begin{align*}
\| \mathcal F^{-1}  \eqref{330b}  \|_{L_x^1} \le C_{\beta} \int_{\gamma}^{\infty}
\tau^{\beta-1} \lambda^{\gamma-\tau} \,d\tau <+\infty,
\end{align*}
where we used the fact that $\lambda>1$.

Finally we deal with the contribution of \eqref{330c}. By \eqref{log_simple}, it is obvious that
\begin{align*}
\left\| \mathcal F^{-1} ( \frac 1 {\log^{\beta} (\lambda+ |\xi| ) } )
\right\|_{L_x^1} = \frac 1 {\log^{\beta} \lambda}
<\infty,
\end{align*}
since $\lambda>1$. Now in \eqref{330c} we may assume $0<\gamma <1$ (the cases $\gamma=0$ and
$\gamma=1$ are trivial). By \eqref{423a} and Young's inequality, we get
\begin{align*}
\| \mathcal F^{-1} \eqref{330c} \|_{L_x^1} & \le \| \mathcal F^{-1} ( \frac 1 {\log^{\beta} (\lambda+|\xi|) } )
\|_{L_x^1} \cdot \| \mathcal F^{-1} ( (\lambda+|\xi|)^{\gamma} - |\xi|^{\gamma} ) \|_{L_x^1} \\
& \le \frac 1 {\log^{\beta} \lambda} \cdot \lambda^{\gamma} <\infty.
\end{align*}

\end{proof}

By a simple substitution $|\xi| \to |\xi|^2$, we can deduce the nonlocal decomposition of
the operator $\A_1$ from Lemma \ref{lem1}. Of course, one still needs to check the
$L^1$ boundedness of the error term under such nonlinear substitution.

\begin{cor}[Nonlocal decomposition for the operator $\A_1$]
\label{cor1}
Let  $d\ge 1$, $0\le \gamma \le 2$ and $\beta> 0$, $\lambda >1$.
 Then we have the
decomposition:
\begin{align} \label{e2000}
\frac { \nb^{\gamma}} {\log^{\beta}(\lambda-\Delta)} = C_{\beta} \int_0^{\gamma}
\tau^{\beta-1} |\nabla|^{ {\gamma} -\tau} \,d\tau + P,
\end{align}
where $P$ is a smooth Fourier multiplier which maps $L^p$ to $L^p$ for all $1\le p \le +\infty$.
More precisely, for any function $f$
\begin{align*}
(Pf)(x) = (K*f)(x) = \int_{\mathbb R^d} K(x-y) f(y)\, dy,
\end{align*}
and $\| K\|_{L_x^1} \le Const$.
\end{cor}

\begin{proof}[Proof of Corollary \ref{cor1}]
On the Fourier side, the identity
\eqref{e2000} is equivalent to the following (the value of the constant $C_{\beta}$ can
be adjusted slightly)
\begin{align} \label{e2000a}
\frac {|\xi|^{\gamma} } { \log^{\beta} ( \lambda +|\xi|^2) }
= C_{\beta} \int_0^{\frac {\gamma}2} \tau^{\beta-1} (|\xi|^2)^{\frac{\gamma} 2 -\tau} \,d\tau +P(\xi).
\end{align}

By using a similar derivation as in the beginning part of the proof of Lemma \ref{lem1} (see
in particular \eqref{log_simple}--\eqref{330c}, and replace $|\xi|$ by $|\xi|^2$, $\gamma$ by $\gamma/2$),
 we obtain \eqref{e2000a} with
\begin{align*}
P(\xi) &=
C_{\beta} \int_0^{\frac {\gamma}2} \tau^{\beta-1}
\Bigl( (\lambda+|\xi|^2)^{\frac {\gamma}2 -\tau} - (|\xi|^2)^{\frac{\gamma}2-\tau} \Bigr)\, d\tau  \\
& \qquad +C_{\beta} \int_{\frac {\gamma}2 }^\infty \tau^{\beta-1} (\lambda+ |\xi|^2)^{\frac {\gamma}2 -\tau}
\,d\tau \\
& \qquad + \frac { (|\xi|^2)^{\frac{\gamma}2} - (\lambda+ |\xi|^2)^{ \frac {\gamma}2} }
{\log^{\beta} ( \lambda +|\xi|^2) }.
\end{align*}
Note that $0\le \frac {\gamma}2 \le 1$ and the fact $\| \mathcal F^{-1} (a+|\xi|^2)^{-s}\|_{L_x^1}
=a^{-s}$ for any $s>0$ and $a>0$.
By using a similar analysis as in the proof of Lemma \ref{lem1}, it is then not difficult to check
that $\mathcal F^{-1} (P)$ is an $L^1$ bounded kernel.
\end{proof}

We now consider the more involved $1<\gamma \le 2$ case for the operator $\A$. One should
compare the decomposition \eqref{e300} with \eqref{e200}.

\begin{lem}[Nonlocal decomposition, case $1< \gamma \le 2$]
\label{lem1a}
Let  $d\ge 1$, $1< \gamma \le 2$ and $\beta> 0$, $\lambda>1$.
 Then we have the
decomposition:
\begin{align} \label{e300}
&\frac { \nb^{\gamma}} {\log^{\beta}(\lambda+|\nabla|)}   \notag \\
=&\;
C_{\beta} \int_{\gamma-1}^{\gamma} \tau^{\beta-1} | \nabla|^{\gamma-\tau} d\tau
+
C_{\beta} \int_0^{\gamma-1}
\tau^{\beta-1} \Bigl( |\nabla|^{ {\gamma} -\tau}
- \lambda \,\tau  |\nabla|^{\gamma-\tau-1} \Bigr)
\,d\tau + P,
\end{align}
where $P$ is a smooth Fourier multiplier which maps $L^p$ to $L^p$ for all $1\le p \le +\infty$.
More precisely, for any function $f$
\begin{align*}
(Pf)(x) = (K*f)(x) = \int K(x-y) f(y)\, dy,
\end{align*}
and $\| K\|_{L_x^1} \le Const$.
\end{lem}

\begin{proof}[Proof of Lemma \ref{lem1a}]
Throughout this proof we shall use the letter $P$ to denote the
symbol of an $L^1 \to L^1$ bounded operator. For the
convenience of notation, we allow the value of $P$ to vary from line
to line. We begin with two elementary estimates. Let $\phi \in
C_c^{\infty} (\mathbb R^d)$ be a radial smooth cut-off function such
that $\phi(x)=1$ for $|x|\le 2$ and $\phi(x)=0$ for $|x|\ge 3$. For
any constant $C>0$ define $\phi_{<C} (x):= \phi( x/C)$ and
$\phi_{>C} (x) = 1- \phi_{<C}(x)$. Then for any $\gamma>0$, $C>0$,
$s\ge 0$, we have the following
\begin{align}
&\| \mathcal F^{-1} \Bigl(  (\frac {|\xi|} {\lambda+|\xi|} )^{\gamma} \Bigr) \|_{L_x^1 \to L_x^1} <\infty, \label{e24a} \\
&\| \mathcal F^{-1} \Bigl( |\xi|^{-s} \phi_{>C}(\xi)  \Bigr) \|_{L_x^1} \le C_1 (1+s)^{d+1} C^{-s}, \label{e24b}
\end{align}
where the constant $C_1$ is independent of $s$.  To prove \eqref{e24a}, one can use a scaling argument to reduce
to the $\lambda=1$ case. The result then follows easily from the binomial expansion of $(|\xi|/(1+|\xi|))^{\gamma}
=(1- \frac 1 {1+|\xi|})^{\gamma}=\sum_{n\ge 0} C_{\gamma, n} (1+|\xi|)^{-n}$,  the $L^1$-boundedness of the operators $(1+|\xi|)^{-n}$
(namely $\| \mathcal F^{-1} ( (1+|\xi|)^{-n} ) \|_{L_x^1\to L_x^1} \le 1$ for any $n\ge 0$),
and the fact that $\sum_{n\ge 0} |C_{\gamma,n}|<+\infty$ (note that $C_{n,\gamma}$ has a definite sign for $n$ sufficiently large).  To prove \eqref{e24b} one can again use scaling to reduce to
the case $C=1$.  For $s \ge 2$ the result is obvious by using integration by parts. For $0<s\le 2$,  we note that
\begin{align*}
\|  \mathcal F^{-1} \Bigl( (1+|\xi|)^{-s} \phi_{>C}(\xi)  \Bigr) \|_{L_x^1} \lesssim 1,
\end{align*}
and on the support of $\phi_{>C}(\xi)$,
\begin{align*}
|\xi|^{-s} - (1+|\xi|)^{-s} \sim |\xi|^{-s-1}.
\end{align*}
One can then
use the Littlewood-Paley operators  to bound (since we are summing over $N\ge N_0$ the convergence in $L^1$ is
of no problem):
\begin{align*}
\|\mathcal F^{-1} ( |\xi|^{-s} \phi_{>1} ) \|_{L_x^1\to L_x^1} &\; \lesssim 1+ \|P_{\gtrsim 1} (  (|\nabla|^{-s} - (1+|\nabla|)^{-s} ) \delta_0 ) \|_{L_x^1} \notag \\
& \lesssim\; 1+ \sum_{\text{N dyadic}: N\gtrsim 1} N^{-s-1}   \| P_N   \delta_0 \|_{L_x^1}   \notag \\
& \le Const.
\end{align*}
Observe that the implied constants are uniform in $s$
since $0<s\le 2$. This settles \eqref{e24b}.

By \eqref{log_simple}, we write
\begin{align}
\Gamma(\beta) \frac {|\xi|^{\gamma}} {\log^{\beta} (\lambda+ |\xi|) } &= \int_0^{\gamma-1} \tau^{\beta-1} |\xi|^{\gamma}
(\lambda+|\xi|)^{-\tau} \,d\tau \label{136a} \\
& \; + \int_{\gamma-1}^{\gamma} \tau^{\beta-1} |\xi|^{\gamma}
(\lambda+|\xi|)^{-\tau} \,d\tau \label{136b} \\
&\;+ \int_{\gamma}^{\infty} \tau^{\beta-1} |\xi|^{\gamma}
(\lambda+|\xi|)^{-\tau} \,d\tau. \label{136c}
\end{align}

We first deal with \eqref{136c}. Rewrite
\begin{align*}
|\xi|^{\gamma}
(\lambda+|\xi|)^{-\tau} = (\frac{|\xi} { \lambda+|\xi|})^{\gamma} \cdot (\lambda+|\xi|)^{\gamma-\tau}.
\end{align*}
By  using \eqref{e24a} we obtain
\begin{align*}
\| \mathcal F^{-1}  \eqref{136c} \|_{L_x^1} \lesssim \int_{\gamma}^\infty  \tau^{\beta-1} \lambda^{\gamma-\tau} \,d\tau <\infty.
\end{align*}

Next we turn to \eqref{136b}.  By inserting a smooth cut-off function $\phi_{>10\lambda} (\xi)$, We have
\begin{align}
\eqref{136b} &= \int_{\gamma-1}^{\gamma}
 \tau^{\beta-1}
  |\xi|^{\gamma-\tau} (1+ \frac {\lambda}{|\xi|} )^{-\tau} \,d\tau \phi_{>10\lambda} (\xi)+P \notag \\
&=\int_{\gamma-1}^{\gamma} \tau^{\beta-1} |\xi|^{\gamma-\tau}  \,d\tau \phi_{>10\lambda} (\xi) + P \notag \\
& \qquad + \int_{\gamma-1}^{\gamma} \tau^{\beta-1} |\xi|^{\gamma-\tau} ( (1+\frac{\lambda} {|\xi|} )^{-\tau} -1 ) \,d\tau
\phi_{>10 \lambda} (\xi) . \label{t11a}
\end{align}

  On the other hand, by using the binomial expansion of the function $(1+t)^{-s} =\sum_{n\ge 0} C_{n,s} t^n $ and the estimate \eqref{e24b},
  it is not difficult to check that
  \begin{align*}
  \| \mathcal F^{-1} \eqref{t11a} \|_{L_x^1} & \lesssim  \int_{\gamma-1}^{\gamma} \tau^{\beta-1}
 \Bigl( \sum_{n\ge 1} |C_{n,\tau}| \cdot \lambda^n \cdot (10\lambda)^{-(n+\tau-\gamma)} \cdot (1+n+\tau-\gamma)^{d+1}
 \Bigr)\,d\tau
         \notag\\
         & <+\infty.
  \end{align*}

  Hence we have proved
  \begin{align*}
  \eqref{136b} = \int_{\gamma-1}^{\gamma} \tau^{\beta-1} |\xi|^{\gamma-\tau} \,d\tau + P.
  \end{align*}

  We turn now to the final piece \eqref{136a}. The main idea is similar to that of \eqref{136b}. We again insert a smooth cut-off
  $\phi_{>10\lambda}(\xi)$ and observe that when $|\xi| \ge 10\lambda$, we have
  \begin{align*}
  (1+\frac {\lambda} {|\xi| } )^{-\tau} = 1- \tau \frac {\lambda} {|\xi|} + \sum_{n \ge 2} C_{n,\gamma} ( \frac {\lambda}{|\xi|} )^n.
  \end{align*}
  Note that $0\le \tau \le \gamma-1$ and we have to keep terms up to the linear term. Then clearly
  \begin{align*}
  \eqref{136a} = \int_0^{\gamma-1} \tau^{\beta-1} ( |\xi|^{\gamma-\tau} - \tau \lambda |\xi|^{\gamma-\tau-1} ) \,d\tau +P.
  \end{align*}

  The desired decomposition  \eqref{e300} now follows.

\end{proof}

\section{Proof of Theorem \ref{thm1} and Theorem \ref{thm2} for the operator $\A_1$}
In this section we give the proofs of Theorems \ref{thm1} and \ref{thm2} for the operator $\A_1$.
The proofs for the operator $\A$ is slightly more involved and will be given in the next section.

\begin{proof}[Proof of Theorem \ref{thm1} for the operator $\mathcal A_1$]
Assume first $0<\gamma\le 2$ and $\beta>0$. By \eqref{e2000} we can write
\begin{align}
\A_1 = L + P, \label{decomp_1}
\end{align}
where $L= C_{\beta} \int_0^{\gamma} \tau^{\beta-1} |\nabla|^{\gamma-\tau} d\tau$ and $\| P\|_{L_x^1} <+\infty$.
Now take $\lambda_1> \nu \| P\|_{L_x^1}$ and define $f(t,x) = e^{-\lambda_1 t }\theta(t,x)$. Fix $T>0$ and consider
\begin{align*}
\sup_{0\le t \le T, \, x \in \mathbb R^d} |f(t,x)| = M>0.
\end{align*}
Without loss of generality we can assume
\begin{align*}
\sup_{0\le t \le T, \, x \in \mathbb R^d} f(t,x) =M>0.
\end{align*}
By using the decay condition \eqref{decay_cond} and a simple compactness argument in $t$, we conclude that
there exists $(t_0,x_0)$ such that
\begin{align*}
f(t_0,x_0) = M.
\end{align*}
We now show that $t_0=0$. Indeed if $0<t_0\le T$, we compute
\begin{align}
(\partial_t f)(t_0,x_0) &= -\lambda_1 f(t_0,x_0) - \nu (\A_1 f)(t_0,x_0) \notag \\
&= - \lambda_1 M - \nu (Lf)(t_0,x_0) - \nu (Pf)(t_0,x_0). \label{F12a}
\end{align}

Now by Corollary \ref{cor1}, we have
\begin{align}
\| Pf(t_0) \|_{L_x^\infty} \le \|P\|_{L_x^1} \| f(t_0) \|_{L_x^\infty} \le \|P\|_{L_x^1} \cdot M. \label{F12b}
\end{align}

For any $0<s<2$, by using the fractional representation
\begin{align*}
(|\nabla|^s g)(x) = C_s \lim_{\epsilon \to 0 } \int_{|y-x|> \epsilon} \frac {g(x) -g(y)} {|x-y|^{d+s} } \,dy,
\end{align*}
it is easy to see $(|\nabla|^s f)(t_0,x_0) \ge 0$ and hence
\begin{align}
(Lf)(t_0,x_0) \ge 0. \label{F12c}
\end{align}

Plugging \eqref{F12b} and \eqref{F12c} into \eqref{F12a}, we reach a
contradiction:
\begin{align*}
(\partial_t f)(t_0,x_0) < - (\lambda_1-\nu \|P\|_{L_x^1} ) M <0.
\end{align*}

Therefore we conclude that $t_0=0$ and clearly the estimate \eqref{emax_1} follows.

It remains to prove the case $\gamma=0$ and $\beta>0$. But in this case the operator
$\log^{-\beta} ( \lambda -\Delta)$ corresponds to an $L^1$-bounded convolution kernel.
Hence we just need to repeat the previous argument with $\A_1=P$ and $\lambda_1>\nu \|P\|_{L_x^1}$.
We omit the repetitive details.

\end{proof}
Finally we complete the
\begin{proof}[Proof of Theorem \ref{thm2} for the operator $\A_1$]
Without loss of generality we assume $\nu>0$, $0<\gamma\le 2$ and $\beta>0$. Let $1\le p<\infty$.
Multiplying both sides of \eqref{e1} by $|\theta|^{p-1} \text{sgn}(\theta)$, integrating by parts and
using the fact that $v$ is divergence free, we obtain
\begin{align}
\frac 1 p \frac {d} {dt} ( \|\theta(t) \|_p^p) & =  -\nu \int_{\mathbb R^d} (\A_1\theta )|\theta|^{p-1}
\text{sgn}(\theta)\,dx
\notag \\
& = - \nu \int_{\mathbb R^d} (L\theta) |\theta|^{p-1} \text{sgn}(\theta) \,dx - \nu \| (P\theta) |\theta|^{p-1} \|_{L_x^1},
\label{1010a}
\end{align}
where in the last equality we have used the decomposition \eqref{decomp_1}.

Since for any $0\le s < 2$, $1\le p<\infty$, we have
\begin{align*}
\int_{\mathbb R^d} (|\nabla|^s \theta)|\theta|^{p-1} \text{sgn}(\theta) \,dx \ge 0,
\end{align*}
it follows easily that
\begin{align*}
\int_{\mathbb R^d} (L \theta) |\theta|^{p-1} \text{sgn}(\theta) \,dx \ge 0.
\end{align*}

By H\"older, we have
\begin{align*}
\| (P\theta) |\theta|^{p-1} \|_{L_x^1} \le \| P\|_{L_x^1} \| \theta\|_{L_x^p}^p.
\end{align*}

Plugging the above estimates into \eqref{1010a} and integrating in time, we get
for any $1\le p <\infty$,
\begin{align*}
\| \theta(t) \|_{L_x^p}\le e^{ \nu \|P\|_{L_x^1} t } \| \theta_0 \|_{L_x^p}.
\end{align*}
The case $p=\infty$ follows by a limiting argument $p\to \infty$.
Clearly \eqref{emax_2} holds by setting $C=\nu \|P\|_{L_x^1}$.
\end{proof}

\section{Proof of the main theorems for the operator $\A$}
In this section we describe the proofs of the main theorems for the operator $\A$. We shall only
need to consider the case $1<\gamma \le 2$.  Thanks to Lemma \ref{lem1} the case $0\le \gamma \le 1$ is
already covered in the previous
section. In the case $1<\gamma \le 2$ we have to use the decomposition \eqref{e300} in Lemma \ref{lem1a}.  The extra
complication is due to the  negative term $-\lambda \tau |\nabla|^{\gamma-\tau-1}$ which in principle can cause
 the maximum principle to fail. The way out of this difficulty is to note that the main term $|\nabla|^{\gamma-\tau}$
 is stronger than this negative term by an order of $|\nabla|^{-1}$. The following lemma quantifies this
 observation. In some sense it gives the maximum principle for "mixed" operators.

 \begin{lem} \label{lem_11}
Let $0<s_1<s_2<2$ and $C_1>0$. Consider the operator
\begin{align*}
L=|\nabla|^{s_2} - C_1 |\nabla|^{s_1}.
\end{align*}
Then for any smooth function $g$ which attains its maximum at some point $x_0$, we have
\begin{align}
(Lg)(x_0) \ge - C_1 C_d \| g\|_{\infty}
 \cdot  (2-s_1)
\Bigl( 1 +  \bigl(  {s_2} (2-s_2)\bigr)^{-\frac{s_1}{s_2-s_1} } \Bigr).
  \label{1100a}
\end{align}
where $C_d>0$ is some constant depending only on the dimension $d$. In particular if $s_1=s_2-1$ and $1<s_2<2$, then we have
the estimate
\begin{align}
(Lg)(x_0) \ge -C_1 C_d^{\prime} \| g\|_{\infty} ( 1+ (2-s_2)^{1-s_2}), \label{1100b}
\end{align}
where $C_d^\prime>0$ is another constant depending only on the dimension $d$.
 \end{lem}

\begin{proof}[Proof of Lemma \ref{lem_11}]
To begin we need to derive the explicit constant appearing
in the integral representation of the fractional operators $|\nabla|^s$ with $0<s<2$.
Recall that for any $0<\alpha<d$, the Riesz potential $|\nabla|^{-\alpha}$ has
the following explicit representation (cf. pp 117 of Stein \cite{Stein1})
\begin{align*}
(|\nabla|^{-\alpha} f) = c_{\alpha,d} \int_{\mathbb R^d} |x-y|^{-d+\alpha} f(y) \,dy,
\end{align*}
with
\begin{align*}
c_{\alpha,d} = \frac {\Gamma( \frac {d-\alpha} 2 ) } { \Gamma ( \frac{\alpha}2 ) 2^{\alpha} \pi^{\frac d 2} }.
\end{align*}

For $0<s<2$ by writing $|\nabla|^{s} = -\Delta |\nabla|^{-(2-s)}$ and integrating by parts, we get
\begin{align*}
(|\nabla|^s f ) = C_{s,d}  \lim_{\epsilon \to 0} \int_{|y-x|>\epsilon } \frac {f(x)-f(y)} {|x-y|^{d+s} } \,dy,
\end{align*}
where
\begin{align*} 
C_{s,d} = s \frac {\Gamma ( \frac {d+s} 2 ) } {2^{2-s} \Gamma( \frac{2-s} 2 ) \pi^{\frac d 2} }.
\end{align*}
By using the asymptotics $\Gamma(z) \sim z^{-1}$ for $z\sim 0$, it is easy to see that
\begin{align}
C_{s,d} \sim_d \frac {s} {\Gamma(\frac {2-s} 2 )} \sim_d s (2-s). \label{const_1}
\end{align}
and in particular for all $0<s<2$,
\begin{align}
C_{s,d} \lesssim_d 1, \label{const_3}
\end{align}

We now write
\begin{align}
\Bigl( (|\nabla|^{s_2} - C_1 |\nabla|^{s_1} ) g \Bigr)(x_0) = C_{s_2,d} \lim_{\epsilon \to 0}
\int_{|y-x_0|> \epsilon} \frac{g(x_0) -g(y) } { |x_0 - y|^{d+s_2} } \,dy \notag \\
-C_1 C_{s_1,d}  \lim_{\epsilon \to 0}
\int_{|y-x_0|> \epsilon} \frac{g(x_0) -g(y) } { |x_0 - y|^{d+s_1} } \,dy. \label{1030a}
\end{align}

Observe that $g(x_0)-g(y) \ge 0$ for all $y$. We now separate the $y$-integral into two regimes.
The first regime is $\{y: |y-x_0| \le \min \{1, C_2  \}\}$, where $C_2>0$ is a constant such that
(here we use \eqref{const_3} to bound $C_{s_1,d}$)
\begin{align*}
(C_2)^{s_1-s_2} C_{s_2,d}  \gtrsim_d C_1.
\end{align*}
By using \eqref{const_1}, we have (the notation $\sim_{C_1,d}$ means up a constant depending on $C_1$ and $d$)
\begin{align*}
C_2 \sim_{C_1,d} \Bigl(  {s_2 } (2-s_2)
\Bigr)^{\frac 1 {s_2-s_1}}.  
\end{align*}

In the first regime, it is easy to check that the first integral bounds the second integral in  \eqref{1030a}.
The second regime is just the complement $\{ y: |y-x_0|> \min \{ 1, C_2 \} \}$. In this case we simply discard
the first integral and bound the second integral by $\| g\|_{L_x^\infty}$ which produces a term of the form
(below $C_d$ denotes a constant depending only on the dimension $d$):
\begin{align*}
&- C_1  C_d\frac {s_1}{\Gamma(\frac {2-s_1} 2 )} \| g\|_{\infty} \int_{r> \min \{1, C_2 \} } r^{-1-s_1}\, dr \notag\\
\gtrsim_{C_1,d} &\quad -\| g\|_{\infty} \cdot  (2-s_1)
\Bigl( 1 +  \bigl(  {s_2} (2-s_2)\bigr)^{-\frac{s_1}{s_2-s_1} } \Bigr).
\end{align*}
This settles \eqref{1100a}.

Finally \eqref{1100b} is a simple consequence of \eqref{1100a}.
\end{proof}

The following corollary will be used in the proof of Theorem
\ref{thm1}.

\begin{cor}  \label{cor11}
Let  $d\ge 1$, $1< \gamma \le 2$ and $\beta> 0$, $\lambda>1$.
Then for any smooth function $g$ which attains its maximum at some point $x_0$, we have
\begin{equation*} 
\Bigl( \frac { \nb^{\gamma}} {\log^{\beta}(\lambda+|\nabla|)} g \Bigr) (x_0)
\ge - (C_{d,\beta,\gamma}  \lambda + \|P\|_{L_x^1} ) \| g\|_{\infty},
\end{equation*}
where $C_{d,\beta,\gamma}$ is some constant depending only on $(d,\alpha,\beta,\gamma)$,
and $P$ is the operator defined in \eqref{e300}.
\end{cor}

\begin{proof}[Proof of Corollary \ref{cor11}]
By Lemma \ref{lem1a}, Lemma \ref{lem_11},
 and the fact that $(|\nabla|^{\gamma-\tau} g)(x_0)\ge 0$,
we have
\begin{align*}
&\Bigl( \frac { \nb^{\gamma}} {\log^{\beta}(\lambda+|\nabla|)} g \Bigr) (x_0)   \notag \\
\ge & \quad C_{\beta} \int_0^{\gamma-1} \tau^{\beta-1}
\Bigl( (|\nabla|^{\gamma-\tau} -\lambda \tau|\nabla|^{\gamma-\tau-1})g \Bigr)(x_0)  d\tau  - \| P\|_{L_x^1} \|g\|_\infty
\notag\\
\ge & \quad - C_{\beta,d} \cdot \lambda \|g\|_\infty
\int_0^{\gamma-1} \tau^{\beta} ( 1 + (2-\gamma+\tau)^{1-\gamma+\tau}
) d\tau  - \| P\|_{L_x^1} \|g\|_\infty
 \notag \\
\ge & -(C_{d,\beta,\gamma} \lambda +\|P\|_{L_x^1} ) \| g\|_{\infty}.
\end{align*}

\end{proof}

We are now ready to complete the
\begin{proof}[Proof of Theorem \ref{thm1} for the operator $\A$, case $1<\gamma\le 2$]
With the help of Corollary \ref{cor11}, the proof is similar to the proof for the operator
$\A_1$ in section 3, one only needs to consider $f(t,x)= e^{-\lambda_1 t} \theta(t,x)$ with
$\lambda_1> \nu (C_{d,\beta,\gamma} \lambda + \|P\|_{L_x^1} )$, where the constant $C_{d,\beta,\gamma}$
is the same as in Corollary \ref{cor11}. The rest of the proof is now the same as in Section 3.
We omit the details.
\end{proof}

Next we turn to the proof of Theorem \ref{thm2}. The following lemma establishes
a form of maximum principle for the mixed operator $L=|\nabla|^s- C_1 |\nabla|^{s-1}$ in
the $L^p$, $1\le p<\infty$ setting.

\begin{lem} \label{lem15}
Let $1<\gamma\le 2$, $1<s<\gamma$, $C_1>0$ and consider the operator
\begin{align*}
L = |\nabla|^s - C_1 (\gamma-s) |\nabla|^{s-1}.
\end{align*}
Then for any $1\le p<\infty$ and any smooth $g \in L^p$, we have the bound
\begin{align}
\int_{\mathbb R^d} (Lg) |g|^{p-1} \sgn(g) \,dx \ge - C_1 C_{d,\gamma}\| g\|_p^p,
\label{458a}
\end{align}
where $C_{d,\gamma}>0$ is some constant depending only on $(d,\gamma)$.
\end{lem}
\begin{proof}[Proof of Lemma \ref{lem15}]
The proof is analogous to that of Corollary \ref{cor11} with suitable modifications.
Without loss of generality we assume $C_1=1$.
We begin with a simple tail estimate. For any $A>0$, we have
\begin{align}
&  (s-1)  \int_{\mathbb R^d} \int_{|y-x|>A}
  \frac{|f(x)-f(y)|}{|x-y|^{d+s-1}} (|f(x)|^{p-1}+|f(y)|^{p-1})  \,dx \,dy\notag \\
\le & C_d \| f\|_p^p \cdot A^{1-s}, \label{tmp504}
\end{align}
where $C_d>0$ is a constant depending only on the dimension $d$.
The estimate \eqref{tmp504} is a simple consequence of the Young's inequality
\begin{align*}
|f(y)| \cdot |f(x)|^{p-1} \le \frac{p-1} p |f(x)|^p + \frac 1 p |f(y)|^p,
\end{align*}
and Fubini.

Next we need to transform the LHS of \eqref{458a} suitably.
By a symmetrization in the variable $x$ and $y$, we have
\begin{align}
&\int_{\mathbb R^d} (|\nabla|^s  g) |g|^{p-1} \sgn(g) \,dx \notag \\
 =&
C_{s,d} \int \int \frac {g(x)-g(y)} {|x-y|^{d+s}}dy |g(x)|^{p-1} \sgn(g(x)) \,dx \notag \\
 =& \frac {C_{s,d}}2 \int \int \frac{h(x,y) }
{|x-y|^{d+s}} \,dx \,dy, \label{521a}
\end{align}
where
\begin{align*}
h(x,y)=(g(x)-g(y))\Bigl (|g(x)|^{p-1} \sgn(g(x)) - |g(y)|^{p-1} \sgn(g(y)) \Bigr).
\end{align*}
Note that for any real numbers $a,b$, we have
\begin{align*}
(a-b)\Bigl( |a|^{p-1} \sgn(a) -|b|^{p-1} \sgn(b) \Bigr) \ge 0.
\end{align*}
Therefore $h(x,y) \ge 0$ for all $x,y$.
The advantage of the expression \eqref{521a} is that the integrand is always non-negative.
Similar expression holds for the operator $|\nabla|^{s-1}$.

Let $A>0$ be a constant whose value will be specified later. By  using \eqref{521a} and \eqref{const_1},
 we have
\begin{align}
\text{LHS of \eqref{458a}} & \ge C_d s(2-s) \int_{\mathbb R^d} \int_{|y-x|<A} \frac {h(x,y)}
{|x-y|^{d+s}} \,dx \,dy \label{535a}\\
& \qquad - C_d^{\prime} (s-1) (\gamma-s) \int_{\mathbb R^d}\int_{|y-x|<A} \frac {h(x,y)}
{|x-y|^{d+s-1}} \,dx \,dy \label{535b}\\
& \qquad - C_d^{\prime} (s-1) (\gamma-s) \int_{\mathbb R^d} \int_{|y-x|>A} \frac{h(x,y)}{|x-y|^{d+s-1}} \,dx \,dy, \label{535c}
\end{align}
where $C_d$ and $C_d^\prime$ are constants depending only on the dimension $d$.

By \eqref{tmp504}, we have
\begin{align}
\eqref{535c} \ge - C_d^{\prime} \| g\|_p^p A^{1-s} (\gamma-s). \label{535cc}
\end{align}

On the other hand, by choosing $A=C_d/C^{\prime}_d$, we have
\begin{align}
C_d \cdot s (2-s) \cdot \frac 1 A \ge C_d^{\prime} (s-1) (\gamma-s),  \notag 
\end{align}
and
\begin{align*}
\eqref{535a}+ \eqref{535b} \ge 0.
\end{align*}
Substituting the value of $A$ into \eqref{535cc}, we obtain \eqref{458a}.

\end{proof}
Finally we are ready to complete the

\begin{proof}[Proof of Theorem \ref{thm2} for the operator $\A$ in the regime $1<\gamma\le 2$]
Thanks to Lemma \ref{lem1a} and Lemma \ref{lem15}, we essentially only have to repeat the proof for the operator
$\A_1$ in Section 3. In place of \eqref{1010a}, we have
\begin{align*}
\frac 1 p \frac d {dt} ( \| \theta(t) \|_p^p )
\le \nu ( \| P \|_{L_x^1} + \lambda C_{d,\gamma,\beta} ) \| \theta (t) \|_p^p.
\end{align*}

The estimate \eqref{emax_2} follows immediately.

\end{proof}

\end{document}